\documentclass[letterpaper, 10 pt, conference]{ieeeconf}  

\IEEEoverridecommandlockouts                              
\usepackage{subcaption}
\usepackage[utf8]{inputenc}
\usepackage[fleqn]{amsmath}
\usepackage{dsfont} 
\usepackage{tabularx,booktabs}
\usepackage{units}
\usepackage{bropd} 
\usepackage{bm}    
\usepackage{currfile}
\usepackage{tikz}
\usetikzlibrary{arrows} 
\usetikzlibrary{calc} 
\usetikzlibrary{positioning}
\usepackage{pgfplots}
\pgfplotsset{compat=1.13}
\pgfplotsset{table/search path={figures//}} 
\usepgfplotslibrary{external}
\newlength\figureheight
\newlength\figurewidth
\newlength\leftlabeldist
\newlength{\bottomlabeldist}
\usepackage{myabbreviations}

\newtheorem{proposition}{Proposition}
\newtheorem{assumption}{Assumption}
\newtheorem{lemma}{Lemma}

\usepackage[absolute,showboxes]{textpos}

\TPMargin{5pt}

\newcommand{\copyrightstatement}{
    \begin{textblock}{0.87}(0.06,0.93)    
        \noindent
        \textcopyright 2021 IEEE.  Personal use of this material is permitted.  Permission from IEEE must be obtained for all other uses, in any current or future media, including reprinting/republishing this material for advertising or promotional purposes, creating new collective works, for resale or redistribution to servers or lists, or reuse of any copyrighted component of this work in other works.
    \end{textblock}
}

\title{\LARGE \bf
    \vspace*{-0.5pt}{Efficient Formulation of Collision Avoidance Constraints in Optimization Based Trajectory Planning and Control}
}

\author{Max Lutz and Thomas Meurer
    \thanks{Max Lutz and Thomas Meurer are with the Chair of Automatic Control, Faculty of Engineering, Kiel University, 24143 Kiel, Germany 
    \mbox{\tt\small \{mlut, tm\}@tf.uni-kiel.de}}%
}

\begin{document}

\maketitle
\copyrightstatement
\thispagestyle{empty}
\pagestyle{empty}

\begin{abstract}
    To be applicable to real world scenarios trajectory planning schemes for mobile autonomous systems must be able to efficiently deal with obstacles in the area of operation.
    In the context of optimization based trajectory planning and control a number of different approaches to formulate collision avoidance constraints can be found in the literature.
    Here the contribution of the present work is twofold.
    First, the most popular methods to represent obstacles are summarized, namely the simple ellipsoidal representation, the constructive solid geometry (CSG) method as well as a direct and an indirect implementation of a signed distance based approach.
    The formulations are characterized with respect to the impact on the complexity of the optimization problem, as well as the ability to meet different problem requirements.
    Second, this work presents a novel variant of the CSG method to describe collision avoidance constraints.
    It is highly efficient due to a very low number of nonlinear inequality constraints required for a given number of obstacles and sample points and in contrast to the original CSG formulation allows to consider the controlled system's shape.
    The good performance of the proposed formulation is demonstrated by a comparison to the previously mentioned alternatives.
    To this end optimal trajectory planning for marine surface vessels formulated as a nonlinear programming problem is used as a benchmark example where the scenario is designed based on the maritime test field in Kiel, Germany. 
\end{abstract}
%
%
%
\section{Introduction}
Obstacle avoidance is an important element of optimization based trajectory planning for all autonomous systems that are not track bound such as vehicles, vessels, aircraft, robots and alike.
Solution techniques that are based on the formulation of an optimal control problem (OCP), subsequent discretization and iterative solution using a nonlinear programming problem (NLP) solver require an explicit mathematical model of the obstacles, i.e., occupied or forbidden subsets of the controlled system's operating space.
A natural choice due to its simplicity is the use of ellipsoids or even spheres, see, e.g., \cite{Chen1988,Geisert2016,Bitar2018}.
There are several important drawbacks associated with this simple modeling technique when realistic scenarios are considered, as discussed, e.g., in \cite{Zhang2020} for the task of parallel parking a car, or for the motion planning for a humanoid robot in \cite{Schulman2014}.
It is thus not surprising that the interest in numerically efficient representations of objects in 2D and 3D goes back to the early research on numerical implementations of optimal control problems, e.g., aiming to solve optimal trajectory planning tasks for robots \cite{Gilbert1984}.
In the quest for an efficient mathematical formulation numerous different methods were developed.
Some, such as the Gilbert-Johnson-Keerthi (GJK) algorithm \cite{Gilbert1988}, a simplex based iterative approach to calculate the distance between two convex sets, have become highly popular in other applications, namely computer graphics and physics simulators (physics engines).
On the other hand a constructive solid geometry (CSG) method with smooth approximation of the intersection and union operations has originally been developed for computer graphics applications in \cite{Ricci1973} but allows to represent obstacles in the OCP context in a very compact way.
Super ellipsoids, i.e., sets described by different p-norms, as used in the optimal trajectory planning for vehicles by, e.g., \cite{Gong2009} and \cite{Febbo2017} are essentially a subset of the shapes attainable by this CSG method.
Examples for the application of this CSG method itself include the formulation of constraint equations for the trajectory planning for underwater robots in \cite{Wang2000} which also discusses the drawbacks of using a potential field for obstacle avoidance as proposed by, e.g., \cite{Khatib1985,Romon1990} in combination with CSG.
In a more recent contribution in the context of autonomous docking of marine surface vessels harbor structures are avoided using CSG based collision avoidance constraints in \cite{Helling2020}.
An interesting discussion on the connections between different approaches to CSG and some applications in various fields are given by \cite{Shapiro1994} and the references therein.
An example for the use of the GJK algorithm is given by the before mentioned humanoid robot motion planning \cite{Schulman2014}, while the collision avoidance constraints for the car parking problem \cite{Zhang2020} are formulated using dual variables.
Due to the vast number of different approaches that have been proposed this short summary is by no means complete and only focuses on some popular examples chosen to serve as a comparison for the novel formulation introduced in this paper.
It is organized as follows.
First the general OCP formulation is outlined in Section~\ref{sec:probFormulation}, following with a recapitulation of popular obstacle modeling approaches in Section~\ref{sec:obstclForm}.
Then in Section \ref{sec:csg_new} the novel CSG based formulation of collision avoidance constraints is introduced.
The benchmark scenario is introduced in Section \ref{sec:results} and simulation results are given.
The paper concludes with some summarizing remarks.
\section{Problem Formulation}
\label{sec:probFormulation}
Often the infinite dimensional OCP expressing the trajectory planning task for a dynamic system is discretized by, e.g., a direct method to yield a NLP of the form
\begin{subequations}
    \label{eq:nlpGeneral}
    \begin{alignat}{2}
        \label{eq:nlp_J}
        \min_{\vec \xi}\; &c\left(\vec \xi\right) = \sum_{k=0}^{N}l_k\left(\vec \xi\right)\hspace{-2ex}\\
        \label{eq:nlp_eqConstr}
        \text{s.t.} \quad &\vec{g}\left(\vec \xi \right)=\vec{0}\,, \\
        \label{eq:nlp_inEqConstr}
        &\vec{h}\left(\vec \xi \right)\leq\vec{0}\,,\,
    \end{alignat}
\end{subequations}
using $\nSamplePts$ equally spaced sample points in time.
The cost function \eqref{eq:nlp_J} with the running costs $l_k(\vec \xi)$, the equality constraints \eqref{eq:nlp_eqConstr} and the inequality constraints \eqref{eq:nlp_inEqConstr} are written as functions of the vector of decision variables $\vec \xi\in\R^\nDecVar$.
In this context obstacle constraints in the operating space $\R^n$ are formulated as inequalities in terms of the decision variables and included in \eqref{eq:nlp_inEqConstr}.
To allow for an efficient solution of the  NLP with state of the art solvers based on iterative techniques, such as sequential quadratic programming or interior point methods, availability of an analytical expression for the gradient of the cost function as well as the nonlinear equality and inequality constraints w.r.t. the decision variables $\vec \xi$ is important.

Due to a lack of space subsequently only static obstacles are considered to facilitate the notation and also the illustration of simulation results.
All approaches presented can easily be extended to accommodate for dynamic obstacle positions by introducing a time dependency in the obstacle modeling equations, the increase in computational effort of this can be expected to be similar for all methods.
Other issues that may arise with dynamic obstacles such as, e.g. the question of recursive feasibility when repeatedly solving the optimization problem in the context of model predictive control, are generally not linked to the choice of obstacle modeling approach and thus also do not affect the latter.

In the following the controlled object's position is denoted by the point $\pX \in \R^n$ if collision avoidance is enforced for a reference point only, where the dependency on the decision variables is stated for clarity.
If the system's shape is considered the occupied area is denoted by $\Vx \subset \R^n$.
\section{Existing Obstacle Formulations}
\label{sec:obstclForm}
\subsection{Spherical and Ellipsoidal Obstacles}
Modeling obstacles as regions with ellipsoidal or as a special case spherical shape is attractive due to the straightforward mathematical representation of the obstacle region
\begin{align}
    \label{eq:ellipsoid}
        \Obstcl=\left\{\vec p\in\R^n|(\vec p - \pO)\T A (\vec p - \pO) \leq 1\right\}\,,
\end{align}
where $\pO \in \R^n$ denotes the obstacle center and $A$ is a positive definite, symmetric matrix defining the shape.
For a number of $\nObstcls$ obstacle regions a total number of $\nObstcls\nSamplePts$ nonlinear inequalities have to be added to the OCP, the calculation of the corresponding gradient is straightforward.
Clearly the assumption that all obstacles or rather all convex sub-regions can be approximated well by an ellipsoid may not hold in a realistic scenario.
Further complications arise if a uniform safety distance around the obstacle is to be respected.
While it is easy to construct concentric spheres, a surface with a uniform distance to an ellipsoid is not an ellipsoid itself \cite{FAROUKI1997}.
Also, if the controlled object's shape is to be taken into account it is not trivial to efficiently calculate a distance between two ellipsoids (unless they are spherical), even in the 2D case.
\subsection{Constructive Solid Geometry}
Some of the shortcomings of the spherical/ellipsoidal obstacle representation can be overcome by making the step to the CSG method proposed by \cite{Ricci1973}.
Central to the method is the non-negative function $f_{\txt{o}}$ called the (characteristic) defining function, a type of real-valued implicit point membership classification (PMC) function \cite{Shapiro1994}. It is used to express geometric shapes in space, i.e., \mbox{$f_{\txt{o}}:\vec p \in \R^n \mapsto \R^+$}, where $f_{\txt{o}} = 1$ defines the boundary, $f_{\txt{o}} < 1$ the interior and $f_{\txt{o}} > 1$ the exterior.
Arbitrarily shaped convex or non-convex regions may be constructed by combining geometric primitives with the intersection and union operations.
In \cite{Ricci1973} the smooth approximations
\begin{subequations}
    \begin{align}
                \label{eq:approxIntersCSG}
                I_p\left(f_{\txt{o},\,1},\dots,f_{\txt{o},\,\nFaces}\right) &= \big(f_{\txt{o},\,1}^p+\dots+f_{\txt{o},\,\nFaces}^p\big)^{\frac{1}{p}}\\
                \intertext{of the intersection and }
        \label{eq:approxUnionCSG}
        U_p\left(f_{\txt{o},\,1},\dots,f_{\txt{o},\,\nFaces}\right) &= \big(f_{\txt{o},\,1}^{-p}+\dots+f_{\txt{o},\,\nFaces}^{-p}\big)^{-\frac{1}{p}}
    \end{align}    
\end{subequations}
of the union operation with $p\in\N$ are proposed.
Here the quality of the approximation increases with $p\to\infty$, for details refer to \cite{Ricci1973}.
Equation \eqref{eq:ellipsoid} gives an example for the defining function of an ellipsoid, if $A$ is allowed to be indefinite it can also express elliptic cylinders and slabs (regions bound by two parallel hyperplanes).

In addition to the improved flexibility w.r.t. the possible obstacle shapes the concept also allows to drastically reduce the number of nonlinear inequality constraints, as irrespective of the number of individual obstacles the application of the approximated union operation allows to construct $f_{\txt{o},\,\Sigma}(\vec \xi) = U_p(f_{\txt{o},\,1},\dots,f_{\txt{o},\,\nObstcls})$, i.e., a single function representing the union over all obstacle regions.
Consequently only a number of $\nSamplePts$ constraints of the form $1-f_{\txt{o},\,\Sigma}(x_k)\leq 0$ for $k = 0,\dots,N$ need to be added to the OCP if the CSG method is used.
While the complexity of the constraint equations increases in comparison to simple ellipsoidal shapes, the analytic gradient can still be easily computed by successive application of the chain rule.
However, the defining function as proposed by \cite{Ricci1973} does not yield a distance measure that can directly be mapped to the euclidean distance so that the introduction of a uniform safety distance remains complicated.
A further drawback of this method is that it is difficult to efficiently consider the controlled object's shape.
A solution by sampling its boundary is given in, e.g., \cite{Helling2020}.
\subsection{Dual Signed Distance Obstacle Representation}
The signed distance function $\sd(\Vx,\Obstcl)$ generalizes the distance between two sets $\Vx, \Obstcl\in\R^n$ by combining the euclidean distance 
\begin{subequations}
    \label{eq:dist+pen}
\begin{align}
    \label{eq:dist}
    \dist(\Vx,\Obstcl) = \!\inf_{\vec p\in\R^n}\{\lVert \vec p \rVert\!:\!(\mathcal{V}(\vec \xi) + \vec p) \cap \Obstcl \neq \emptyset\}   
\end{align}
and the penetration     
\begin{align}
    \label{eq:pen}
    \pen(\Vx,\Obstcl) = \!\inf_{\vec p\in\R^n}\{\lVert\vec p\rVert\!:\!(\mathcal{V}(\vec \xi)+ \vec p)\cap \Obstcl = \emptyset\}
\end{align}
\end{subequations}
to 
\begin{align}
    \label{eq:sd}
    \sd(\Vx,\Obstcl) = \dist(\Vx,\Obstcl) - \pen(\Vx,\Obstcl)\,.
\end{align}
Simply speaking the signed distance extends the notion of distance to the case where objects penetrate, in which case it is equal to the negative euclidean distance either object needs to be moved for them to be out of contact.
To calculate the signed distance \eqref{eq:sd} generally the optimization problems \eqref{eq:dist} and \eqref{eq:pen} have to be solved, as such there is no explicit representation of the signed distance in terms of the decision variables $\vec \xi$.
For compact convex sets however, strong duality holds and an implementation using dual variables is proposed by \cite{Zhang2020} in the context of optimal control.
To simplify notation in the following shapes are limited to convex polytopes, a limitation that results in a negligible restriction regarding the ability to express various practically relevant obstacle shapes.

A polyhedral obstacle $\Obstcl$ with $\nFaces$ faces in $\R^n$ can elegantly be described by 
\begin{align}
    \label{eq:obstclPolygon}
    \Obstcl = \left\{\vec p\in \R^n | A\vec p - \vec b\leq \vec 0\right\}\,,
\end{align}
using the face unit normal vectors $\vec n_i$, $i=1,\dots,\nFaces$ to build the matrix $A=[\vec n_1,\dots,\vec n_\nFaces]\T\in\R^{\nFaces\times n}$ and arbitrary points $\vec o_i$, $i =1,\dots,\nFaces$ on the hyperplanes containing the faces to form the vector $\vec b = [\vec n_1\T \vec o_1,\dots,\vec n_K\T\vec o_\nFaces]\T\in\R^\nFaces$.
As the formulation is based on the signed distance it is straightforward to introduce a minimum distance $d_{\txt{min}}$ to each obstacle.
For the point system case with position $\pX$ the collision-avoidance constraints may then be expressed as
\begin{align}
    \begin{split}
        \sd(\{\pX\},\Obstcl) \geq d_{\txt{min}} \;\Leftrightarrow\; \exists \vec \mu&\geq\vec{0}:\\
        \vec \mu\T\left(A\pX - \vec{b} \right) &\geq d_{\txt{min}}\,,\\
        \lVert A\T\vec \mu\rVert &= 1
    \end{split}
\end{align}
using the dual variables $\vec \mu\in \R^\nFaces$. 
If the shape of the controlled system is to be considered, additional dual variables $\vec \lambda\in\R^\nFacesSys$ have to be used, where $\nFacesSys$ is the number of faces of the polyhedron representing the controlled system.
The matrix with the corresponding face unit normal vectors is denoted by $V(\vec \xi)=[\vec n_1(\vec \xi),\dots,\vec n_\nFacesSys(\vec \xi)]\T\in\R^{\nFacesSys\times n}$.
The components of the vector $\vec q(\vec \xi) = [\vec n_1\T(\vec \xi) \vec v_1(\vec \xi),\dots,\vec n_\nFacesSys\T(\vec \xi)\vec v_\nFacesSys(\vec \xi)]\T\in\R^\nFacesSys$ are given by the inner product of the face unit normal vector and a point $\vec v_k(\vec \xi)$, $k =1,\dots,\nFacesSys$ that lies on the hyperplane containing the corresponding face.
The collision avoidance constraints then read
\begin{align}
    \begin{split}
        \sd(\Vx,\Obstcl) \geq d_{\txt{min}} \;\Leftrightarrow \;\exists \vec \mu&\geq\vec{0}\,,\; \vec \lambda\geq\vec{0}:\\
        -\vec{\lambda}\T\vec q(\vec \xi)-\vec \mu\T\vec b &\geq d_{\txt{min}}\,,\\
        V\T(\vec \xi)\vec{\lambda}+A\T\vec \mu&=\vec{0}\,,\\
        \lVert A\T\vec \mu\rVert&=1\,.
    \end{split}
\end{align}
For the derivation of these dual implementations and the formulation for the case of general convex sets the interested reader is referred to \cite{Zhang2020} and the references therein.

Obviously the dual signed distance based obstacle representation comes at the cost of a large number of equality and inequality constraints that have to be added to the problem and an additional increase of the problem dimension due to the introduction of the dual variables.
For the point system case with $\nSamplePts$ sample points and $\nObstcls$ obstacles $2\nSamplePts\nObstcls$ nonlinear equality and inequality constraints and $\nSamplePts\sum_{i=1}^{\nObstcls} \nFaces_i$ additional dual decision variables result.
If the shape of the controlled system is to be considered the increase in dimension and number of constraints is even more drastic, with further $n\nSamplePts\nObstcls$ linear equality constraints, where $n$ is the dimension of the operating space, and $\nSamplePts\nObstcls\nFacesSys$ additional dual decision variables.
\subsection{Direct Signed Distance Obstacle Representation}
\label{subsec:gjk}
The dual signed distance obstacle representation uses a reformulation of the signed distance problem \eqref{eq:dist+pen} in terms of dual variables and subsequently solves the problem with the NLP solver used to tackle the OCP itself, yet there are dedicated algorithms specifically developed to solve the signed distance problem.
As mentioned above a direct implementation of a signed distance based obstacle formulation has to deal with the lack of an explicit equation of the signed distance in terms of the decision variables.
When the GJK algorithm \cite{Gilbert1988} and the expanding polytope algorithm (EPA) \cite{van2003collision} are used to calculate the distance and the penalty, respectively, they also provide the separation/penetration vector, i.e., $\vec p$ in \eqref{eq:dist} and \eqref{eq:pen}.
This does allow the calculation of the proper gradient in the current point in non-degenerate situations, as also used by \cite{Schulman2014}.
Degenerate situations are encountered when the closest features of two shapes are not unique points but parallel faces, compare \cite[Fig. 4]{Schulman2014}.
The approach is dismissed as a local linearization with uncertain error bounds by \cite{Zhang2020}, but this is not inherent to the method itself but rather caused by the way it is implemented in the context of a sequential convex optimization scheme by \cite{Schulman2014}.
Essentially, except for degenerate situations, from a numerical point of view the calculation of nonlinear collision avoidance constraints with the GJK and EPA algorithms can be handled in the same way as any nonlinear constraint in NLP, where the gradient in the current point, i.e., the local linearization, is provided to the solver.
In line with the findings of \cite{Schulman2014} degenerate situations do not appear to be a problem in practical application and, in any case, do not occur if the controlled system's shape is neglected.
To include the collision avoidance constraints in the OCP for $\nObstcls$ arbitrarily shaped but convex obstacles only $\nSamplePts\nObstcls$ inequalities of the form $d_{i,\,\txt{min}}-\sd(\vec\cdot,\Obstcl_i)\leq 0$ for $i=1,\dots,\nObstcls$ have to be considered, where $\vec\cdot$ may be replaced with $\{\pX\}$ or $\Vx$.
However, the computational cost of running the GJK and, in case of penetration, the EPA algorithm has to be taken into account.
\section{Novel CSG-based Obstacle Avoidance Constraint Formulation}
\label{sec:csg_new}
In this section a novel formulation based on the CSG concept for the obstacle avoidance constraints in an NLP describing an optimization based trajectory planning and control problem is outlined.
First the general idea is presented for the simple case in which the collision avoidance is only enforced for the controlled system's reference point, then it is discussed how the concept can be extended to consider the controlled system's shape.
\begin{assumption}
    \label{ass:unionPrim}
    An obstacle can be represented as a set $\Obstcl \in\R^n$ that results from the intersection of a finite number $\nFaces$ of primitives $\Prim_{i} \in \R^n$, $i=1,\dots,\nFaces$, i.e., $\Obstcl = \bigcap_{i=1}^{\nFaces} \Prim_{i}$.
\end{assumption}
\begin{lemma}[\cite{Shapiro1994}]
    \label{lem:signedDistance}
    The signed distance $\sd(\{\pX\},\Obstcl)$ is a valid PMC function to model objects with CSG.
\end{lemma}
\begin{proposition}
    \label{prop:PointSys}
    When Assumption \ref{ass:unionPrim} holds the signed distance $\sd(\{\pX\},\Obstcl)$ is bound from below by
    \begin{align}
        \label{eq:proposedPointVessel}
        \max_{\Prim_{i}}\left\{\sd(\{\pX\},\Prim_{i})\right\} \leq \sd(\{\pX\},\Obstcl)\,,
    \end{align}
    where the controlled system's position is $\pX$.
\end{proposition}
\begin{proof}
    As $\forall i\in \{1,\dots,\nFaces\} : \Obstcl\subseteq\Prim_{i}$ obviously $\dist(\{\pX\},\Prim_{i}) \leq \dist(\{\pX\},\Obstcl)$ and $\pen(\{\pX\},\Prim_{i}) \geq \pen(\{\pX\},\Obstcl)$ follow. With the definition of the signed distance \eqref{eq:sd} Proposition \ref{prop:PointSys} holds.
\end{proof}
Clearly this is advantageous whenever it is easy to determine the signed distance to the individual primitives, as, e.g., if they are half-spaces, balls and alike.
\begin{assumption}
    \label{ass:unionPrimSys}
    Let the controlled system's shape be described by the set $\Vx\in\R^n$, resulting from the intersection of a finite number $\nFacesSys$ of primitives $\Prim_{\mathcal{V},\,j}(\vec \xi)\in\R^n$, $j=1,\dots,\nFacesSys$, i.e., $\Vx = \bigcap_{j=1}^{\nFacesSys} \Prim_{\mathcal{V},\,j}(\vec \xi)$.
\end{assumption}
To be able to consider the controlled system's shape the Minkowski difference, defined for two sets $\Vx,\Obstcl\in\R^n$
\begin{align}
    \label{eq:defMinkDiff}
    \Obstcl-\Vx := \{(\vec o - \vec v)\in \R^n|\vec o \in \Obstcl \land \vec v \in \Vx\}
\end{align}
is adopted, a concept that the signed distance based methods outlined above also make use of internally.
\begin{lemma}[\cite{Cameron1986}]
    \label{lem:minkDiff}
    $\sd(\Vx,\Obstcl) = \sd(\vec 0,\Obstcl-\Vx)$
\end{lemma}
Lemma~\ref{lem:minkDiff} reduces the problem of finding the signed distance between two sets to finding the signed distance between a point and a set.
\begin{lemma}
    The Minkowski difference distributes over the set intersection, i.e., for sets $\mathcal{A}, \mathcal{B}, \mathcal{C} \in \R^n$ it holds \mbox{$(\A \cap \B) - \C = (\A-\C) \cap (\B -\C)$}.
\end{lemma}
\begin{proof}
    Consider $\vec x\in(\A-\C)\cap(\B-\C)$ and $\vec c \in \C$. Clearly $\vec x \in (\A-\C) \land \vec x \in (\B-\C)$. With \eqref{eq:defMinkDiff} $\exists \vec a\in\A,  \vec b\in\B: \vec x = \vec a - \vec c \land \vec x = \vec b - \vec c$ $\Rightarrow \vec a = \vec b \in \A \cap \B$. Thus applying \eqref{eq:defMinkDiff} again $(\A\cap\B)-\C \supseteq (\A-\C)\cap(\B-\C)$ follows.
    
    Now take $\vec y \in (\A\cap\B)-\C$. Again using \eqref{eq:defMinkDiff} $\exists \vec z \in \A \cap \B, \vec c \in \C: \vec y = \vec z - \vec c.$ Now with $\vec z \in \A \cap \B \Rightarrow \exists \vec a \in \A, \vec b \in \B: \vec a = \vec b = \vec z$. Thus with \eqref{eq:defMinkDiff} $\vec y \in \A-\C \land \vec y \in \B-\C$ and $(\A-\C)\cap(\B-\C) \supseteq (\A\cap\B)-\C$.
\end{proof}
\begin{proposition}
    If Assumption \ref{ass:unionPrim} and Assumption \ref{ass:unionPrimSys} hold the signed distance $\sd(\Vx,\Obstcl)$ is bound from below by
    \begin{align}
        \label{eq:proposedVesselOutline}
        \begin{split}
            \max_{\Prim_{i},\,\Prim_{\mathcal{V},\,j}}&\left\{\sd(\vec 0,\Obstcl -\Prim_{\mathcal{V},\,j}(\vec \xi)),\sd(\vec 0,\Prim_{i}- \Vx\right\} \\
        \leq &\sd(\Vx,\Obstcl)\,,
    \end{split}
    \end{align}
    with $i\in\{1,\dots,\nFaces\}$ and $j\in\{1,\dots,\nFacesSys\}$.
\end{proposition}
\begin{proof}
    Under Assumptions \ref{ass:unionPrim} and \ref{ass:unionPrimSys}, using the Minkowski difference and its distributive property w.r.t. the set intersection the equalities $\sd(\Vx,\Obstcl) = \sd(\vec 0,\Obstcl-\Vx) = \sd(\vec 0, \Obstcl - \bigcap_{j=1}^\nFacesSys\Prim_{\mathcal{V},\,j}) = \sd(\vec 0,\bigcap_{j=1}^\nFacesSys(\Obstcl - \Prim_{\mathcal{V},\,j}))$ hold. In the same manner $\sd(\Vx,\Obstcl)= \sd(\vec 0,\bigcap_{i=1}^\nFaces(\Prim_i - \Vx))$ results. Following the proof of Proposition~\ref{prop:PointSys} $\max_{\Prim_{\mathcal{V},\,j}}\sd(\vec 0,\Obstcl-\Prim_{\mathcal{V},\,j}) \leq \sd(\vec 0,\bigcap_{j=1}^\nFacesSys(\Obstcl - \Prim_{\mathcal{V},\,j}))$ and  $\max_{\Prim_{i}}\sd(\vec 0,\Prim_i-\Vx) \leq \sd(\vec 0,\bigcap_{i=1}^\nFaces(\Prim_i - \Vx))$ follow.
\end{proof}

Informally speaking, similar to the point system case the lower bound on the signed distance between the set occupied by the system to an obstacle is given by the maximum of the signed distance between the origin and the geometric primitives forming the Minkowski difference of the system and the obstacle.
For the gradient the same arguments as made in subsection  \ref{subsec:gjk} hold regarding degenerate situations. 
Also, the maximum function yields a non-smooth gradient elsewhere.
However, if a smooth gradient is required for all non-degenerate situations the obstacle avoidance constraints can be implemented using a smooth approximation of the maximum function.
The LogSumExp (LSE) function $\lse : \R^m \mapsto \R$,
\begin{align}
    \label{eq:logSumExp}
    \lse(d_1,\dots, d_m) &= \frac{1}{\alpha}\log\!\bigg(\!\sum_{i=1}^m\!\alpha(d_i-d_0)\!\bigg)\! +d_0\,,\!
\end{align}
is a stable solution and well established in the machine learning and artificial intelligence community.
Here $\alpha\in\{\R\backslash0\}$ determines the quality of the approximation, with $\lse(d_1,\dots,d_m) \to \max(d_1,\dots,d_m)$ for $\alpha \to \infty$ and $\lse(d_1,\dots,d_m) \to \min(d_1,\dots,d_m)$ for $\alpha \to -\infty$.
The parameter $d_0$, typically chosen to be $d_0 = \max(d_1,\dots,d_m)$ for $\alpha>0$ or $d_0 \min(d_1,\dots,d_m)$ for $\alpha<0$, is introduced for reasons of numerical stability.
For a given $\alpha>0$ the approximation error is bound according to 
\begin{align}   
    \begin{split}
    &\max(d_1,\dots,d_m)\leq\lse(d_1,\dots,d_m)\\
    \leq& \max(d_1,\dots,d_m) + \frac{1}{\alpha}\log(m)\,,
\end{split}
\end{align}
the equivalent holds for $\alpha <0$.  
The availability of an upper bound on the approximation error allows to introduce a small safety distance $d_\txt{min} = \nicefrac{1}{\alpha}\log{m}$ to ensure that despite the use of the smooth approximation no collision occurs in any point.

As a CSG based formulation it is possible to implement the collision avoidance constraints using the union over all obstacles, possibly in terms of the smooth approximation of the minimum \eqref{eq:logSumExp} depending on the requirements of the solver.
This reduces the required number of collision avoidance constraints needed to $\nSamplePts$, the number of sample points, irrespective of the number of (sub-)obstacles $\nObstcls$.
The same is true for the simple ellipsoidal obstacle representation and the direct signed distance obstacle representation from subsection \ref{subsec:gjk}, as both are valid PMC functions.

\section{Comparative Simulation Study}
\label{sec:results}

As a benchmark problem the energy optimal point to point trajectory planning for underactuated autonomous surface vessels as presented in \cite{Lutz} is considered.
There it is demonstrated how using a flatness based direct method this task may be expressed as a finite dimensional NLP.
In the following first key elements of the approach are outlined, then the layout of the comparative study of the collision avoidance constraint formulations discussed above is described. 
The former is intentionally kept brief as the focus of the present work is on the latter.

The OCP is formulated using the nonlinear dynamics
\begin{subequations}
    \label{eq:dynVesselModel}
    \begin{align}
        \vdot{\eta} &= R(\psi) \vec{\nu}\\
        M\vdot{\nu} &= -\br{C(\vec{\nu})+D(\vec{\nu})}\vec{\nu}+ \vec{\tau}\,, 
    \end{align}
\end{subequations}
where the vessel pose in the north-east-down (NED) frame is denoted by $\vec \eta$. 
The body fixed velocities $\vec \nu$ are calculated based on the effect of the external forces and moments $\vec \tau = [\tau_u, \tau_v, \tau_r]\T$.
Here $M$ is the inertia matrix, $C(\vec \nu)$ the coriolis matrix and $D(\vec \nu)$ the damping matrix including nonlinear damping terms.
Actuator limits and actuator rate limits are considered in the form of box constraints.

To realize energy optimal trajectory planning the cost function uses $l(t) = (\nicefrac{\tau_u}{\tau_{u,\,\txt{max}}})^2 + (\nicefrac{\tau_r}{\tau_{r,\,\txt{max}}})^2$ as the running cost, assuming that the quadratic form adequately approximates instantaneous energy consumption of the actuators.

For the fully actuated case the dynamics \eqref{eq:dynVesselModel} are differentially flat, with a flat output $\vec z$ given by the position and pose of the system in the NED frame, i.e., $\vec z = \vec \nu$.
Thus exact flat parametrizations $\vec \theta_{\vec x}(\vec z, \vdot z, \vddot z)$ and $\vec \theta_{\vec \tau}(\vec z, \vdot z, \vddot z)$ of the system states and inputs as functions of the flat output and its time derivatives exist \cite{Fliess1995}.
This allows to write all constraints in terms of the flat output and its derivatives, as well as the cost function, which then implicitly encodes the system dynamics.
Considering a parametrized ansatz function for the flat output yields a finite dimensional NLP  of the form \eqref{eq:nlpGeneral}, where the decision variables are the parameters of the ansatz function.
To apply to the underactuated case the sway force $\tau_v$ is constrained to zero by means of an equality constraint and the control input is given by the surge force $\tau_u$ and the yaw moment $\tau_r$. 
For details and parameter values see \cite{Lutz} and the references therein.

The benchmark scenario is built to replicate the setting of the real world maritime test field in Kiel, Germany at $\unit[54^\circ19'40'']{N}, \unit[10^\circ09'54'']{E}$.
The setting is re-scaled as tests using a small unmanned surface vessel (USV) are planned.
For further information on the USV model the reader is referred to \cite{Do2006}.
A local NED coordinate system is defined with the origin centered in the origin of the body fixed coordinate system of the vessel at $t_0 = 0$.
The scenario is defined such that the vessel is at rest initially, $x_0 = [0,0,\unit[-135]{^\circ},0,0,0]\T$, and a trajectory to leave the harbor area is planned with $\te=\unit[60]{s}$ and $x_\txt{e}=[0,\unit[-8.5]{m},0,\unitfrac[0.3]{m}{s},0,0]\T$.
The OCP is solved with a total number of 61 equally spaced sample points using the NLP solver SNOPT \cite{GilMS05}.
To allow for an equal comparison the harbor structures are enclosed in 3 ellipses when using the ellipsoidal formulation, in 3 rounded rectangular shapes for the original CSG formulation, compare \cite[eq. (14)]{Lutz}, or in 3 polygons with a total number of 18 vertices for all other cases.
The vessel shape is described by a symmetrical polygon with a length of \unit[1.2]{m} and a width of \unit[0.36]{m} and 5 vertices, corresponding to the dimensions of the USV, see Fig.~\ref{fig:NEDframe}.
\begin{figure}[!t]
    \vspace*{2mm}
    \captionsetup[subfigure]{aboveskip=-1pt,belowskip=0pt}
    \begin{subfigure}[t]{\linewidth}
        \centering{
            \includegraphics{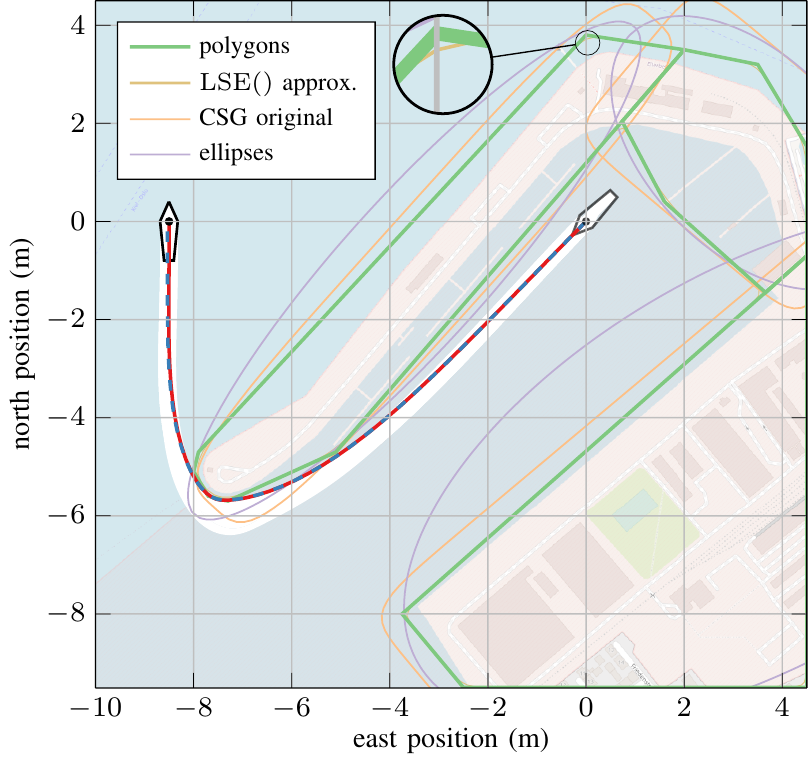}}
        \caption{Local NED frame with vessel and all obstacle shapes used for the simulation study. Additionally the swept volume of the optimization result for the proposed formulation \eqref{eq:proposedVesselOutline} using a hard maximum and the union over all obstacles taking system's shape into account is given (white).}
        \label{fig:NEDframe}
    \end{subfigure}%
    \begin{subfigure}[t]{\linewidth}
        \centering{
            \includegraphics{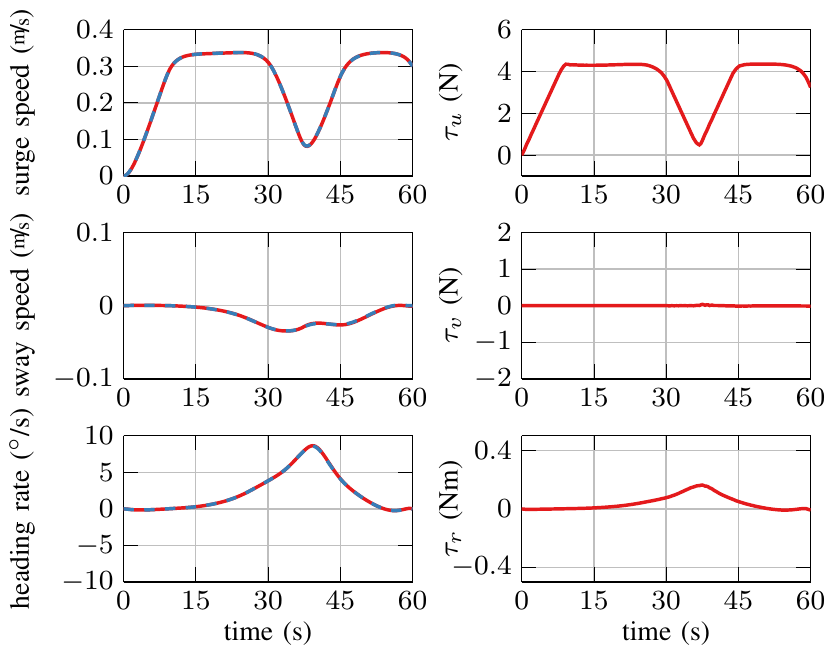}}
        \caption{Body fixed speeds and control inputs.}
        \label{fig:speedsInputs}
    \end{subfigure}%
    \caption{Benchmark scenario overview including optimization results (red) and and simulation results using \matlab's \texttt{ode45()} (blue, dashed) for the proposed formulation \eqref{eq:proposedPointVessel} using a hard maximum and the union over all obstacles.}
\end{figure}
\begin{table*}[!t]
    \centering
    \vspace*{2mm}
    \renewcommand{\arraystretch}{1.1}
    \caption{Results of the simulation study. The novel formulation for the collision avoidance constraints according to Section \ref{sec:csg_new} is denoted by 'proposed hard $\max()$' for \eqref{eq:proposedPointVessel} and 'proposed $\lse()$' if the approximation \eqref{eq:logSumExp} is used.}
    \begin{subtable}[b]{\textwidth}
        \caption{Point system case.}
        \label{tab:pointVessel}
        \begin{tabularx}\textwidth{l c rr c rr c r c rr c rr}
            \toprule
            & \hspace*{0ex} & \multicolumn{2}{c}{proposed hard $\max()$} & \hspace*{0ex} & \multicolumn{2}{c}{proposed $\lse()$} & \hspace*{0ex} & \multicolumn{1}{c}{\multirow{2}{*}{ellipsoidal}} & \hspace*{0ex} & \multicolumn{2}{c}{CSG original} & \hspace*{0ex} & \multicolumn{2}{c}{signed distance}\\
            \cmidrule{3-4} \cmidrule{6-7} \cmidrule{11-12} \cmidrule{14-15}
            && union & separate && union & separate && && union & separate && dual & direct \\
            \midrule
            obstacle constraints            && 61           & 183       && 61       & 183       && 183  && 61       & 183   &&  366     &183\\
            optimization variables          && 189          & 189       && 189      & 189       && 1189  && 189      & 189   && 1287     &189\\
            constraint function calls       && 31           & 35        && 25       & 34        && 40   && 34       & 26    && 81       &47\\
            max. constraint violation       && 1.17e-15     & 2.66e-15  && 6.39e-2  & 6.39e-2   && 0    && 0        & 0     && 1.78e-15 &3.3e-15\\
            cost function value             && 38.22        & 38.27     && 37.58    & 37.58     && 42.91&& 40.78    & 40.78 && 38.62    &38.15\\
            run time solver only [s]         && 1.29         & 1.39      && 1.93     & 1.48      && 1.38 &&  1,37    & 1.46  && 2,28     &1.67\\
            run time total [s]               && 1.47         & 1.66      && 2.08     & 1.78      && 1.68 && 1.60     & 1.68  && 4.46     &2.02\\
            run time total (relative)        && 1            & 1.23      && 1.41     & 1.21      && 1.14 && 1.09     & 1.14  && 3.18     &1.37\\
            \bottomrule
        \end{tabularx}
    \end{subtable}
    \vspace{0.3cm}

    \begin{subtable}[b]{\textwidth}
        \centering
        \caption{System shape considered.}
        \label{tab:vesselOutline}
        \begin{tabular}{l c rr c rr c rr}
            \toprule
            & \hspace*{0ex} & \multicolumn{2}{c}{proposed hard $\max()$} & \hspace*{0ex} & \multicolumn{2}{c}{proposed $\lse()$} &  \hspace*{0ex} & \multicolumn{2}{c}{signed distance}\\
            \cmidrule{3-4} \cmidrule{6-7} \cmidrule{9-10}
            && union & separate && union & separate && dual & direct \\
            \midrule
            obstacle constraints            && 61           & 183       && 61       & 183       && 732      &183\\
            optimization variables          && 189          & 189       && 189      & 189       && 2202     &189\\
            constraint function calls       && 22           & 24        && 82       & 26        && 123      &23\\
            max. constraint violation       && 0            & 3.03e-10  && 0.11     & 0.11      && 2.6e-8   &0\\
            cost function value             && 40.18        & 40.18     && 38.90    & 38.90     && 39.92    &40.18\\
            run time solver only [s]         && 0.95         & 1.29      && 4.61     & 1.22      && 220      &1.26\\
            run time total [s]               && 1.55         & 1.90      && 7.13     & 2.02      && 248      &1.52\\
            run time total (relative)        && 1            & 1.23      && 4.60     & 1.30      && 160      &0.98\\
            \bottomrule
        \end{tabular}
    \end{subtable}
\end{table*}
The polygons are described in the form of \eqref{eq:obstclPolygon}, essentially using half planes as geometric primitives for the novel CSG based method.
For reference the body fixed speeds and control inputs are given in Fig.~{\ref{fig:speedsInputs}.
Run times for a \matlab implementation on a laptop computer with Intel Core i5-6200U CPU clocked at \unit[2.30]{GHz} for the different methods of specifying the collision avoidance constraints are reported in Table~\ref{tab:pointVessel} for the point system case and in Table~\ref{tab:vesselOutline} for the methods that allow to consider the system's shape.
All reported times are the average over ten runs.
The maximum constraint violation specifies the maximum collision avoidance constraint violation of the optimization output and is calculated for the actual obstacle shapes, not the smoothed approximations using the $\lse()$-function \eqref{eq:logSumExp}.
The number of obstacle constraint function calls and the run time of the solver itself, i.e., the time required to solve the problem not taking into account the time spent in user defined functions, are given to demonstrate that the effect of the code efficiency of the user defined functions on the qualitative assessment of the results is negligible.

The results show that the novel CSG based obstacle avoidance constraint formulation is able to efficiently encode the required constraints and that it does perform very well in both the point system case and also considering the system dimension.
This is especially obvious in comparison to the ellipsoidal method, which additionally suffers regarding the optimality of the results due to the limitations in shaping the obstacle constraints.
While for the original CSG based method and the formulation using a hard maximum it appears advantageous to reduce the number of collision avoidance constraints by taking the union over all obstacles, for the smoothed case using the $\lse$-function this is not the case, most likely as the benefit of a reduction on constraints is outweighed by the increased nonlinearity.
In all cases an educated initial guess using the \aStar based concept described in \cite[Sec. III-B]{Lutz} is used, including the dual variables.
Nevertheless the dual signed distance obstacle formulation performs poorly for the point system case and scales very badly as evident when considering the system's shape.
The direct signed distance implementation on the other hand benefits from the efficiency of the highly refined CSG and EPA algorithms for determining the signed distance of two convex shapes, as highlighted by the small difference between solver run time and total run time for a similar number of user function calls as the proposed CSG based method when considering the system dimension.
It is to be expected that this advantage becomes more important with increasing number and complexity of the shapes, however there is still significant room for improvement in the implementation of the proposed CSG based method by using iterative methods to calculate \eqref{eq:proposedPointVessel} and \eqref{eq:proposedVesselOutline}.
It appears that the initial guess is slightly better in the case where the system shape is considered, as solver iterations (indicated by user function calls) and solution times do not reflect the increased complexity.
\addtolength{\textheight}{-10.7cm}   
\section{Conclusion}
A novel method to formulate the collision avoidance constraints for OCPs of mobile autonomous systems is proposed that combines the strengths of commonly used concepts by building on the CSG method.
For the two cases, neglecting system's shape as well as considering it, it is proven that the collision avoidance constraints may be implemented using a lower bound on the signed distance to individual obstacles, a key difference to existing CSG based formulations that do not make the connection to the signed distance and therefore struggle to consider the controlled system's shape.
Furthermore, it is demonstrated how the use of the $\lse()$-function yields smooth gradients while confining the approximation error to strict bounds.
Using a flatness based direct method comparative simulation results are reported for the proposed novel formulation as well as established methods based on ellipsoids, direct and indirect implementations of the signed distance collision avoidance constraints and the original CSG based formulation using a positive defining function.
To this end a realistic scenario based on the real world maritime test area in Kiel, Germany is addressed.
The versatility and efficiency of the novel concept is indicated for the point system case as well as for the case taking into account the system shape.
Current research is focused on validation of the simulation results with a small scale USV.
\bibliographystyle{IEEEtran}
\bibliography{obstacleModelling_lit}

\end{document}